\documentclass[a4paper,10pt, reqno]{amsart}

\usepackage{hyperref}
\hypersetup{
    colorlinks=true,       
    linkcolor=blue,          
    citecolor=blue,        
    filecolor=blue,      
    urlcolor=blue           
}
\numberwithin{equation}{section}
\usepackage{rotating}
\usepackage[all]{xy}
\usepackage{tikz,graphicx}
\usepackage{amssymb}
\usepackage{booktabs}
\usepackage[english]{babel}
\usepackage{multirow}
\usepackage{multicol}
\usepackage{hyperref}
\usepackage{rotating}
\usepackage{xcolor}

\theoremstyle{plain}
\newtheorem{thm}{Theorem}[section]
\newtheorem{lem}[thm]{Lemma}
\newtheorem{prop}[thm]{Proposition}
\newtheorem{cor}[thm]{Corollary}

\theoremstyle{definition}

\newtheorem{example}[thm]{Example}

\newcommand{\Z}{\mathbb{Z}}

\def\Aut{\operatorname{Aut}}
\def\PGL{\operatorname{PGL}}

\usepackage{enumerate}

\begin{document}
\title[On Pseudo-real finite subgroups in $\PGL_3(\mathbb{C})$]{On pseudo-real finite subgroups of \boldmath{$\PGL_3(\mathbb{C})$}}
\author[E. Badr] {E. Badr}
\address{$\bullet$\,\,Eslam Badr}
\address{Mathematics Department,
Faculty of Science, Cairo University, Giza-Egypt}
\email{eslam@sci.cu.edu.eg}
\address{Mathematics and Actuarial Science Department (MACT), American University in Cairo (AUC), New Cairo-Egypt}
\email{eslammath@aucegypt.edu}

\author[A. El-Guindy] {A. El-Guindy}
\address{$\bullet$\,\,Ahmad El-Guindy}
\address{Mathematics Department,
Faculty of Science, Cairo University, Giza, Egypt}
\email{aelguindy@sci.cu.edu.eg}


\keywords{Projective linear groups; Field of moduli; Fields of definitions; Pseudo-real; Smooth plane curves; Automorphism groups}

\subjclass[2020]{20G20, 14L35, 14H37, 22F50.}

\maketitle
\begin{abstract}
Let $G$ be a finite subgroup of $\PGL_3(\mathbb{C})$, and let $\sigma$ be the generator of $\operatorname{Gal}(\mathbb{C}/\mathbb{R})$. We say that $G$ has a \emph{real field of moduli} if $^{\sigma}G$ and $G$ are $\PGL_3(\mathbb{C})$-conjugates. Furthermore, we say that $\mathbb{R}$ is \emph{a field of definition for $G$} or that \emph{$G$ is definable over $\mathbb{R}$} if $G$ is $\PGL_3(\mathbb{C})$-conjugate to some $G'\subset\PGL_3(\mathbb{R})$. In this situation, we call $G'$ \emph{a model for $G$ over $\mathbb{R}$}. On the other hand, if $G$ has a real field of moduli but is not definable over $\mathbb{R}$, then we call $G$ \emph{pseudo-real}.

In this paper, we first show that any finite cyclic subgroup $G=\Z/n\Z$ in $\operatorname{PGL}_3(\mathbb{C})$ has {a real field of moduli} and we provide a necessary and sufficient condition for $G=\Z/n\Z$ to be definable over $\mathbb{R}$; see Theorems \ref{cyclicp2}, \ref{cor12v}, and \ref{cyclicp1v31}. We also prove that any dihedral group $\operatorname{D}_{2n}$ with $n\geq3$ in $\operatorname{PGL}_3(\mathbb{C})$ is definable over  $\mathbb{R}$; see Theorem \ref{dihedralv1}. Furthermore, we study all other classes of finite subgroups of $\operatorname{PGL}_3(\mathbb{C})$, and show that all of them except $\operatorname{A}_4,\,\operatorname{A}_5$ and $\operatorname{S}_4$ are pseudo-real; see Theorems \ref{hessianthm} and \ref{typeCorD}. Finally, we explore the connection of these notions in group theory with their analogues in arithmetic geometry; see Theorem \ref{connection} and Example \ref{exam1v1}. As a result, we can say that if $G$ is definable over $\mathbb{R}$, then its Jordan constant $J(G)=1,2,3, 6$ or $60$.
\end{abstract}
\tableofcontents
\section{Introduction}
The projective general linear group over the complex numbers $\operatorname{PGL}_3(\mathbb{C})$ is widely studied in several branches of mathematics for many reasons. Some of these motivations come from algebraic geometry, arithmetic geometry, and also from group theory. We give some examples of such motivations.

(1) In complex algebraic geometry,  $\operatorname{PGL}_3(\mathbb{C})$ can be viewed as the automorphism group $\Aut(\mathbb{P}^2(\mathbb{C}))$ of the complex projective plane $\mathbb{P}^2(\mathbb{C})$, see \cite[Example 7.1.1]{Har} for example. Moreover, any isomorphism between two smooth complex plane curves $C$ and $C'$ of a fixed degree $d\geq4$ is induced by an element of $\operatorname{PGL}_3(\mathbb{C})$, see \cite[Theorem 1]{Cha78}. For such a curve we have the finiteness result $|\Aut(C)|<+\infty$ due to Hurwitz \cite{Hur92}, hence we can view $\Aut(C)$ as a finite subgroup of $\operatorname{PGL}_3(\mathbb{C})$ acting on a non-singular plane model $F(X,Y,Z)=0$ for $C$ inside $\mathbb{P}^2(\mathbb{C})$.
It is thus natural to classify finite subgroups $G$ in $\operatorname{PGL}_3(\mathbb{C})$. Based on geometrical methods, Mitchell \cite{Mit} achieved such classification. Recently, Harui \cite{Harui} made Mitchell's classification more precise under the assumption that $G=\Aut(C)$ for some smooth plane curve $C$. However, some of these groups live in a short exact sequence, hence group extension problems arise, which can sometimes be hard to solve.

Another parallel line of research is to obtain the stratification of $\mathbb{C}$-isomorphism classes of smooth plane curves of a fixed degree $d$ by their automorphism groups.  Henn in his PhD dissertation \cite{Henn} and Komiya-Kuribayashi \cite{Kur1} accomplished this task for smooth quartic curves ($d=4$), Badr-Bars \cite{MR3508302, Badr2, Badr3} for smooth quinitcs ($d=5$) and for smooth sextics ($d=6$).

(2) In complex arithmetic geometry, the problem of studying fields of definition versus fields of moduli for a Riemann surface $\mathcal{S}$ has attracted a lot of recent research.  For example, we refer to \cite{Artebani, Artebani1, BadrHidalog, Earle, Hidalgo1, Hidalgo3, Hidalgo2, Hugg1, Kontogeorgis}.

More precisely, a subfield $K$ of $\mathbb{C}$ is called \emph{a field of definition for $\mathcal{S}$} if there exists a model
of $\mathcal{S}$ defined by polynomials with coefficients in $K$. \emph{The field of moduli for $\mathcal{S}$} is the intersection of all fields of definition for $\mathcal{S}$. The work of Koizumi \cite{Koizumi} guarantee the existence of a model for $\mathcal{S}$ over a finite extension of its field of moduli. In this direction, the surface $\mathcal{S}$ is said to be \emph{pseudo-real} if its field of moduli is a subfield of $\mathbb{R}$, but $\mathcal{S}$ does not have $\mathbb{R}$ as a field of definition.

The above aspects from algebraic geometry and arithmetic geometry are the main motivation for us to extend the notions of fields of definition, fields of moduli, pseudo-real, to the study of arithmetic groups. Indeed, there has been other instances in which it has been fruitful to translate concepts from arithmetic geometry  to group theory, as we illustrate next.

(3) In group theory, we can measure to which extent an infinite group $\Gamma$ is similar to an abelian group by computing its \emph{Jordan constant}, denoted by $J(\Gamma)$. It is defined to be the smallest positive integer such that any finite subgroup of $\Gamma$ has an abelian normal
subgroup with index not exceeding $J(\Gamma)$. This definition originated
from the theory of abelian varieties, more specifically, \cite[Definition 2.1]{Popov}.

Concerning the Jordan constant $J(\operatorname{PGL}_3(K))$, where $K$ is a field of characteristic $0$, Hu \cite{YHu} showed that it assumes only one of the values: $360,\,168,\,60,\,24,\,12,\,6$, depending on certain conditions satisfied by the field $K$.
In particular, $J(\operatorname{PGL}_3(\mathbb{C}))=360$, see \cite[Theorem 1.2]{YHu} for full details. Also, we will see in this paper that a finite subgroups $G$ in $\operatorname{PGL}_3(\mathbb{C})$ is definable over $\mathbb{R}$ only if $J(G)=1,2,3,6$ or $60$.

\vspace{0.2cm}
\noindent\textbf{Notation.} Throughout the paper, we use the following notations.

\begin{itemize}
  \item $\operatorname{Norm}(G,\PGL_3(\mathbb{C}))$ is the normalizer of $G$ inside $\PGL_3(\mathbb{C})$,
  \item $\zeta_n=e^{\frac{2\pi i}{n}}$, a fixed primitive $n$th root of unity in $\mathbb{C}$.

  \item We shall view $\mathbb{C}^\times$ as a subgroup of $\operatorname{GL}_3(\mathbb{C})$ by identifying $c\in \mathbb{C^\times}$ with $\operatorname{diag}(c,c,c)$. If $A$ is in $\operatorname{GL}_3(\mathbb{C})$, we let $\pi(A)$ denote its image under the canonical projection onto $\PGL_3(\mathbb{C})$, namely $\pi(A)$ is the coset (or equivalence class) $\mathbb{C}^\times A$. To ease notation, we occasionally continue to use $A$ in place of $\pi(A)$ when the context is clear.

  \item If $A=(a_{i,j}) \in \operatorname {GL}_3(\mathbb{C})$, then the projective linear transformation $\pi(A)\in\PGL_3(K)$ is sometimes written as
$$[a_{1,1}X+a_{1,2}Y+a_{1,3}Z:a_{2,1}X+a_{2,2}Y+a_{2,3}Z:a_{3,1}X+a_{3,2}Y+a_{3,3}Z].$$
According to \cite{Mit} (see also \cite{Harui}), an \emph{homology} of period $n$ is a projective linear transformation of the plane $\mathbb{P}^2(\mathbb{C})$, which is $\operatorname{PGL}_3(\mathbb{C})$-conjugate to
$\operatorname{diag}(1,1,\zeta_{n})$. Such a transformation fixes point-wise a projective line $\mathcal{L}$, its axis, and a point $P\in\mathbb{P}^2(\mathbb{C})-\mathcal{L}$, its center. In its canonical form, the line is $\mathcal{L}:Z=0$ and the point is $P=(0:0:1)$. Otherwise, it is a \emph{non-homology}.

\item The Galois group $\operatorname{Gal}(\mathbb{C}/\mathbb{R})$ action on $\operatorname{PGL}_3(\mathbb{C})$ is a left action, denoted by $^{\sigma}\phi$ for any $\phi\in\operatorname{PGL}_3(\mathbb{C})$.

\item For $c\in\mathbb{C}$, $\Re(c)$ and $\Im(c)$ denote the real and the imaginary parts of $c$ respectively, and $|c|$ denotes the absolute value of $c$.
\item We use the formal GAP library notations ``GAP$(n, m)$`` to refer the finite group of order $n$ that appears
in the $m$-th position of the atlas for small finite groups \cite{Gap}. See also \href{https://people.maths.bris.ac.uk/~matyd/GroupNames/index500.html}{GroupNames} \cite{groupnames}.
\item We use the terms “definable over $\mathbb{R}$", “descends to $\mathbb{R}$", and “$\mathbb{R}$ is a field of definition" interchangeably.
\item By $\mathcal{G}_9$ we mean the group generated by $E_{3,1,-1}:=\operatorname{diag}(1,\zeta_3,\zeta_3^{-1})$ and $T:=[Y:X:Z]$ . Clearly, $\mathcal{G}_9$ is a $(\Z/3\Z)^2$ in $\operatorname{PGL}_3(\mathbb{C})$.

By $\mathcal{G}_{12}$ we mean the group generated by $T,\,E_{2,1,0}:=\operatorname{diag}(1,-1,1)$ and $E_{2,0,1}:=\operatorname{diag}(1,1,-1)$. The group generated by $\mathcal{G}_{12}$ and $R:=[X:Z:Y]$ is denoted by $\mathcal{G}_{24}$. Viewing $\mathcal{G}_{12}$ and $\mathcal{G}_{24}$ as automorphism subgroups of the Fermat's quartic curve $\mathcal{F}_{4}\,:\,X^4+Y^4+Z^4=0$ whose full automorphism group equals $\operatorname{GAP}(96,64)=(\Z/4\Z)^2\rtimes\operatorname{S}_3$, we easily identify $\mathcal{G}_{12}$ (resp. $\mathcal{G}_{24}$) with a copy of $\operatorname{A}_4$ (resp. $\operatorname{S}_4$) inside $\operatorname{PGL}_3(\mathbb{R})$.
\item Following the notations in \cite[Chp. 2]{Hugg1}, a type $\mathfrak{C}$-group $G$ is a finite subgroup generated in $\operatorname{PGL}_3(\mathbb{C})$ by $T$ and finitely many diagonal elements of the form $\operatorname{diag}(1,\zeta_n^a,\zeta_n^b)$ for some $n,a,b$. For example, $\mathcal{G}_9$ and $\mathcal{G}_{12}$ are type $\mathfrak{C}$-groups.

A type $\mathfrak{D}$-group $G$ is a subgroup generated in $\operatorname{PGL}_3(\mathbb{C})$ by $R:=[X:Z:Y]$ and a group of type $\mathfrak{C}$.

A type $\mathfrak{P}$-group $G$ is one of the finite (primitive) subgroups in $\operatorname{PGL}_3(\mathbb{C})$ namely, the Hessian groups $\operatorname{Hess}_*$, for $*=216,\,72,\,36$, the alternating groups $\operatorname{A}_*$, for $*=5,\,6$, and the Klein group $\operatorname{PSL}(2,7)$ of order $168$.

The subgroup of all \emph{intransitive} elements, i.e. elements of the shape
    $$\left(
                                              \begin{array}{ccc}
                                                 1 & 0 & 0\\
                                                0 & \ast & \ast\\
                                                0 & \ast & \ast \\
                                              \end{array}
                                            \right),$$
is denoted by $\operatorname{PBD}(2,1)$. The natural group homomorphism $$\left(
                                              \begin{array}{ccc}
                                                 1 & 0 & 0\\
                                                0 & \ast & \ast\\
                                                0 & \ast & \ast \\
                                              \end{array}
                                            \right)\in\operatorname{PBD}(2,1)\mapsto\left(
                                                                                                     \begin{array}{cc}
                                                                                                       \ast & \ast \\
                                                                                                       \ast& \ast \\
                                                                                                     \end{array}
                                                                                                   \right)\in \operatorname{PGL}_2(K)$$
                                                                                                   is denoted by $\Lambda.$

A type $\mathfrak{I}$-group is a finite subgroup of $\operatorname{PBD}(2,1)$.
\end{itemize}

\vspace{0.5cm}
We remark that the finite subgroups of $\operatorname{PGL}_3(\mathbb{C})$ are classified in \cite[Chap. VII]{Miller} (see also \cite[Lemma 2.3.7]{Hugg1}). More precisely, we have that:
\begin{thm}\label{shorter2.3.7}
Up to $\operatorname{PGL}_3(\mathbb{C})$-conjugation, a finite subgroup in $\operatorname{PGL}_3(\mathbb{C})$ is one of the types $\mathfrak{I}$, $\mathfrak{C}$, $\mathfrak{D}$ and $\mathfrak{P}$. Moreover, each type $\mathfrak{P}$-group admits a single conjugacy class in $\operatorname{PGL}_3(\mathbb{C})$, where a representation of each class is provided in section \ref{typePgroups}.
\end{thm}

Obviously, we always can view any type $\mathfrak{I}$-group $G$ as a finite subgroup in $\operatorname{PGL}_2(\mathbb{C})$ through the homomorphism $\Lambda$. On the other hand, the classification of finite subgroups in $\operatorname{PGL}_2(K)$, where $K$ is any field (not necessarily algebraically closed) of characteristic $0$, is known due to A. Beauville \cite{Beauville}. In particular, the question about definability over $\mathbb{R}$ for groups of type $\mathfrak{J}$-groups was addressed. This in turn motivates us to only consider finite subgroups $G$ in $\operatorname{PGL}_3(\mathbb{C})$ of one of the types $\mathfrak{C}$, $\mathfrak{D}$ and $\mathfrak{P}$ and to investigate the property of being pseudo-real for each of them.
\section{Main results}
Let $G\subset\PGL_3(\mathbb{C})$ be cyclic of order $1<n< +\infty$. Up to $\PGL_3(\mathbb{C})$-conjugation, such $G$ is generated by a diagonal element $A:=\operatorname{diag}(1,\zeta_n^a,\zeta_n^b)$,
for some $0\leq a < b\leq n-1$ such that $\operatorname{gcd}(a,b)=1$.
\begin{thm}\label{cyclicp2}
Let $G=\langle A\rangle\subset\PGL_3(\mathbb{C})$ be a cyclic group of order $n$ as above. Then, we have that
\begin{enumerate}[(1)]
\item {$G$ always has a real field of moduli}.
\item $\mathbb{R}$ is a field of definition for $G$ if and only if $A$ and $A^{-1}$ are conjugates via a transformation of the shape $\phi\,\,^{\sigma}\phi^{-1}$ for some $\phi\in\PGL_3(\mathbb{C})$. In this situation, $\phi^{-1}\,G\,\phi$ would give a model for $G$ over $\mathbb{R}$.
\end{enumerate}
\end{thm}
\begin{proof}
See section \S \ref{cyclic}, page 5.
\end{proof}

In particular, we have:
\begin{thm}\label{cor12v}
Let $G=\langle A\rangle\subset\PGL_3(\mathbb{C})$ be a cyclic group of order $n$ as above. Then, there exists a model for $G$ over $\mathbb{R}$ if and only if $n=2$ or $n>2$ such that $a+b,\,a-2b$ or $2a-b$ equals $0\,\operatorname{mod}\,n$. In particular, any cyclic group generated by a homology of period $n\geq3$ is pseudo-real.

Furthermore, we can get a model for $G$ over $\mathbb{R}$ generated by
$$
\phi^{-1}\,A\,\phi=\left(
                                                      \begin{array}{ccc}
                                                       2\Im(\alpha\,\overline{\beta}) & 0 & 0 \\
                                                        0 & 2\Im(\alpha\,\overline{\beta}\,\zeta_n^a) & 2|\beta|^2\,\operatorname{sin}(2\pi a/n) \\
                                                        0 & -2|\alpha|^2\,\operatorname{sin}(2\pi a/n) & 2\Im(\alpha\,\overline{\beta}\,\zeta_n^{-a}) \\
                                                      \end{array}
            \right)$$
            for some $\alpha,\,\beta\in\mathbb{C}^*$
\end{thm}
\begin{proof}
See section \S \ref{cyclic}, page 6.
\end{proof}
The above results can be reformulated using characteristic polynomials of lifts to $B\in \operatorname{GL}_3(\mathbb{C})$. If we denote the characteristic polynomial of such $B$ by $f_B(t)$, then it is straightforward to see that for $c\in \mathbb{C}^*$
\begin{equation}\label{char1}
f_{cB}(t)=c^3f_B(t/c).
\end{equation}
So while we can not attach a single polynomial as a characteristic polynomial to an element $A\in \operatorname{PGL}_3(\mathbb{C})$, we can attach to such an $A$ an equivalence class of polynomials in $\mathbb{C}[t]$ coming from the action given by \eqref{char1}. Such classes are preserved under conjugation in $\operatorname{PGL}_3(\mathbb{C})$, and we can prove the following result.
\begin{cor}\label{cyclicp1v31}
A finite cyclic group $G$ of order $n\geq3$ is definable over  $\mathbb{R}$ if there exists $A \in \operatorname{GL}_3(\mathbb{C})$ such that $\pi(A)$ (the image of $A$ in $\operatorname{PGL}_3(\mathbb{C})$ under the natural projection) generates $G$ in $\operatorname{PGL}_3(\mathbb{C})$  and the  characteristic polynomial $f_A(t)\in \mathbb{R}[t]$.

The converse is not necessarily true.
 \end{cor}
\begin{proof}
See section \S \ref{cyclic}, page 7.
\end{proof}

For $G=\operatorname{D}_{2n}$, a dihedral group in $\operatorname{PGL}_3(\mathbb{C})$, we prove:

\begin{thm}\label{dihedralv1}
Any dihedral group $\operatorname{D}_{2n}$ of order $2n$ with $n\geq3$ in $\operatorname{PGL}_3(\mathbb{C})$ is conjugate to $\langle B,\,A\rangle$, where $B=[X:Z:Y]$ and $A=\operatorname{diag}(1,\zeta_n^a,\zeta_n^{-a})$ for some integer $a$ such that $\operatorname{gcd}(n,a)=1$. Moreover, we always can descend it to {$\mathbb{R}$} as $\langle\,\phi^{-1}\,B\,\phi,\,\phi^{-1}\,A\,\phi\rangle$, where $\phi^{-1}\,A\,\phi$ is as given in Theorem \ref{cor12v} and
$$
\phi^{-1}\,B\,\phi=\left(
                                                      \begin{array}{ccc}
                                                        2\Im(\alpha\,\overline{\beta}) & 0 & 0 \\
                                                        0 & -2\Im(\alpha\,\beta) & -2\Im(\beta^2) \\
                                                        0 & 2\Im(\alpha^2) & 2\Im(\alpha\,\beta) \\
                                                      \end{array}
                                                    \right)\\
$$
for some $\alpha,\beta\in\mathbb{C}^*$.
  \end{thm}
\begin{proof}
See section \S \ref{sectiondihedral}, page 8.
\end{proof}
When $G$ is a type $\mathfrak{P}$-group, we show the following.
\begin{thm}\label{hessianthm}
Any type $\mathfrak{P}$-group $G$ in $\operatorname{PGL}_3(\mathbb{C})$ is pseudo-real except when $G$ equals $\operatorname{A}_5$. More concretely, $\operatorname{A}_5$ has a real field of moduli, and we always can descend it to $\mathbb{R}$ as $
\phi^{-1}\,\langle\,A,\,B,\,C\rangle\,\,\phi$ such that $\phi^{-1}\,A\,\phi$ and $\phi^{-1}\,B\,\phi$ are as described in Theorem \ref{dihedralv1} with $n=5$ and $a=4$, and $\phi^{-1}\,C\,\phi$ equals
$$
\left(
                                                      \begin{array}{ccc}
                                                       4\Im(\alpha\,\overline{\beta}) & 8\Im(\alpha\,\overline{\beta})\,\Re(\alpha) & 8\Im(\alpha\,\overline{\beta})\,\Re(\beta) \\
                                                        2\Im(\overline{\beta}) &2\left(\operatorname{cos}(4\pi/5)\Im(\alpha\overline{\beta})-\operatorname{cos}(2\pi/5)\Im(\alpha\beta)\right) & -2\operatorname{cos}(2\pi/5)\Im(\beta^2) \\
                                                        2\Im(\alpha) & 2\operatorname{cos}(2\pi/5)\Im(\alpha^2) & 2\left(\operatorname{cos}(4\pi/5)\Im(\alpha\overline{\beta})+\operatorname{cos}(2\pi/5)\Im(\alpha\beta)\right) \\
                                                      \end{array}
                                                    \right),
$$
for some $\alpha,\beta\in\mathbb{C}^*$.
\end{thm}
\textbf{\begin{proof}
See section \S \ref{typePgroups}, pages 8-10.
\end{proof} }

When $G$ is a type $\mathfrak{C}$ or a type $\mathfrak{D}$  group, we show the following.
\begin{thm}\label{typeCorD}
Any type $\mathfrak{C}$-group in $\operatorname{PGL}_3(\mathbb{C})$ is pseudo-real, except $\langle T\rangle$ and $\mathcal{G}_{12}$.

Similarly, any type $\mathfrak{D}$-group in $\operatorname{PGL}_3(\mathbb{C})$ is pseudo-real, except $\langle T, R\rangle$ and $\mathcal{G}_{24}$.
\end{thm}
\begin{proof}
See section \S \ref{TypesCD}, page 10-11.
\end{proof}
 A connection with these notions in arithmetic geometry is described by the next result.
\begin{thm}\label{connection}
Let $C:F(X,Y,Z)=0$ be a smooth plane curve over $\mathbb{C}$. If $C$ has a real field of moduli in the Complex Arithmetic Geometry sense (See (2) in the Introduction), then its automorphism group  $\operatorname{Aut}(C)$ has a real field of moduli in the Group Theory sense described in this paper.
\end{thm}
\begin{proof}
See section \S \ref{counterexample111}, page 11-12.
\end{proof}
The converse of Theorem \ref{connection} is not necessarily true. For instance, the generic smooth plane curve $\mathcal{C}$ of a fixed degree $d\geq 4$ over $\mathbb{C}$ has trivial automorphism group. In the language of Riemann surfaces, $\mathcal{C}$ can not be conformally equivalent to its conjugate $^{\sigma}\mathcal{C}$, which means that $\mathcal{C}$ does not have a real field of moduli. (We would like to thank the anonymous referee for suggesting the existence of these types of examples, and we are also grateful to  R. Hidalgo for clarifying some details regarding them).

We also provide the following counter example with non-trivial automorphism group.
\begin{example}\label{exam1v1}
There are infinitely many smooth plane quintic curves defined over $\mathbb{C}$ by an equation of the form
$$
\mathcal{C}_{\alpha,\beta}: X^5+Y^5+Z^5+\alpha X(YZ)^2+\beta X^3(YZ)=0,$$
such that the automorphism group $\Aut(\mathcal{C}_{\alpha,\beta})=\operatorname{D}_{10}$ has { a real field of moduli}, but $\mathcal{C}_{\alpha,\beta}$ does not have { a real field of moduli} as its field of moduli.
\end{example}
\section{The cyclic group case}\label{cyclic}
Suppose that $G=\langle\operatorname{diag}(1,\zeta_n^a,\zeta_n^b)\rangle$ in $\PGL_3(\mathbb{C})$ such that $0\leq a<b\leq n-1$ and $\operatorname{gcd}(a,b)=1$.

Since the complex conjugation automorphism $\sigma:\mathbb{C}\rightarrow\mathbb{C}$ sends $\zeta_n\mapsto\zeta_n^{-1}$, then $^{\sigma}G=\langle\operatorname{diag}(1,\zeta_n^{-a},\zeta_n^{-b})\rangle\,=\,G$. In particular, { $G$ has a real field of moduli}. This proves Theorem \ref{cyclicp2}-(1).
To prove Theorem \ref{cyclicp2}-(2), we assume that $G$ descends to $\mathbb{R}$. That is, there exists $\phi\in\PGL_3(\mathbb{C})$ satisfying  $\phi^{-1}\,A\,\phi\in\PGL_3(\mathbb{R}),$
where $A=\operatorname{diag}(1,\zeta_n^a,\zeta_n^b)$. This holds if and only if
$$
\phi^{-1}\,A\,\phi=\,^{\sigma}\left(\phi^{-1}\,A\,\phi\right)\,=\,^{\sigma}\phi^{-1}\,A^{-1}\,^{\sigma}\phi,
$$
which we can read in two different ways. First as
$$
\left(\phi\,\,^{\sigma}\phi^{-1}\right)^{-1}\,A\,\left(\phi\,\,^{\sigma}\phi^{-1}\right)=A^{-1},
$$
which shows that $A$ and $A^{-1}$ are conjugates via $\phi\,\,^{\sigma}\phi^{-1}$. Second as
$$
\phi^{-1}\,A\,\phi\,=\,^{\sigma}\left(\phi^{-1}\,A\,\phi\right),
$$
which shows that $\phi^{-1}\,A\,\phi\in\operatorname{PGL}_3(\mathbb{R})$ as claimed.

We need the following lemma to discuss Theorem \ref{cor12v}.
\begin{lem}\label{lem1}
Assume $A$ and $B$ are matrices in $\operatorname{GL}_3(\mathbb{C})$ such that $\pi(A)$ and $\pi(B)$ are  $\PGL_3(\mathbb{C})$-conjugates (where $\pi$ denotes the natural projection from $\operatorname{GL}_3(\mathbb{C})$ to $\operatorname{PGL}_3(\mathbb{C})$), then there is a constant $c\in\mathbb{C}^\times$ such that the eigenvalues of $B$  are precisely $c\nu_1,c\nu_2,c\nu_3$, where $\nu_1,\nu_2,\nu_3$ are the eigenvalues of $A$.
\end{lem}

\begin{proof}
  Suppose that there exists $\psi\in\PGL_3(\mathbb{C})$ such that $\psi^{-1}\,\pi(A)\,\psi=\pi(B)$ in $\PGL_3(\mathbb{C})$. Then, this equation corresponds to
  $\psi^{-1}\,A\,\psi=(1/c)B$ in $\operatorname{GL}_3(\mathbb{C})$ for some $c\in\mathbb{C}^\times$. Hence, $A$ and $(1/c)B$ are similar matrices in $\operatorname{GL}_3(\mathbb{C})$, so {by elementary linear algebra, we guarantee that their characteristic polynomials have the same roots, say $\nu_1,\nu_2,\nu_3$} . Therefore, the eigenvalues of $B$ are $c\nu_1,c\nu_2,c\nu_3$.
\end{proof}

We now present the  proof of Theorem \ref{cor12v}.
\begin{proof}(of the necessity direction)
First, assume that $G$ is generated by a homology $A=\operatorname{diag}(1,1,\zeta_n)$.
Since $\{c,c,c\,\zeta_n\}\neq\{1,1,\zeta_n^{-1}\}$ for any $c\in\mathbb{C}^*$ unless $n=2$, then $A$ and $A^{-1}$ are never $\PGL_3(\mathbb{C})$-conjugates for $n\geq 3$ by Lemma \ref{lem1}. In particular, $G$ does not have a model over $\mathbb{R}$ by Theorem \ref{cyclicp2}.

Secondly, assume that $G$ is generated by a non-homology $A=\operatorname{diag}(1,\zeta_n^a,\zeta_n^b)$ such that  $\{c,c\,\zeta_n^a,c\,\zeta_n^b\}=\{1,\zeta_n^{-a},\zeta_n^{-b}\}$ for some $c\in\mathbb{C}^*$. Then, $c$ is either $1,\zeta_n^{-a}$ or $\zeta_n^{-b}$. Moreover,

- if $c=1$, then $\zeta_n^{a}=\zeta_n^{-a},\,\zeta_n^{b}=\zeta_n^{-b}$ or $\zeta_n^{a}=\zeta_n^{-b}$.  That is, $2a=2b=0\,\operatorname{mod}\,n$ or $a+b=0\,\operatorname{mod}\,n$. We discard the case $2a=2b=0\,\operatorname{mod}\,n$ as it implies that $n$ or $n/2$ would divide $\operatorname{gcd}(a,b)=1$, a contradiction because $n\geq3$. This leaves us with $a+b=0\,\operatorname{mod}\,n$.

- if $c=\zeta_n^{-a}$, then $\zeta_n^{b-a}=\zeta_n^{-b}$, and $n\,|\,a-2b=0\,\operatorname{mod}\,n$.

- if $c=\zeta_n^{-b}$, then $\zeta_n^{a-b}=\zeta_n^{-a}$, and $2a-b=0\,\operatorname{mod}\,n$.

This completes the necessity part.
\end{proof}

\begin{proof}(of the sufficiency direction) If $G$ is cyclic generated by a homology of period $2$, then $G$ is $\operatorname{PGL}_3(\mathbb{C})$-conjugate to $\langle\operatorname{diag}(1,1,-1)\rangle$ in $\operatorname{PGL}_3(\mathbb{R})$, and we are done. Otherwise, $G$ is generated by a non-homology $A=\operatorname{diag}(1,\zeta_n^{a},\zeta_n^b)$ of order $n\geq3$ such that $a+b,\,a-2b$ or $2a-b$ equals $0\,\operatorname{mod}\,n$. First, we show that any of the last two situation can be reduced to the first one. Indeed, if $A=\operatorname{diag}(1,\zeta_n^{2b},\zeta_n^b)$, then one can take $\psi=[Y:Z:X]$ so that $$\psi^{-1}\,A\,\psi=\operatorname{diag}(\zeta_n^{b},1,\zeta_n^{2b})=\operatorname{diag}(1,\zeta_n^{-b},\zeta_n^{b})=\operatorname{diag}(1,\zeta_n^{a'},\zeta_n^{-a'})\,\,\text{in}\,\,\operatorname{PGL}_3(\mathbb{C}),$$
where $a':=-b$. Similarly, if $A=\operatorname{diag}(1,\zeta_n^{a},\zeta_n^{2a})$, then take $\psi=[Z:X:Y]$ to get  $$\psi^{-1}\,A\,\psi=\operatorname{diag}(\zeta_n^{a},\zeta_n^{2a},1)=\operatorname{diag}(1,\zeta_n^{a},\zeta_n^{-a})\,\,\text{in}\,\,\operatorname{PGL}_3(\mathbb{C}).$$
Now we are going to handle the situation when $n$ divides $a+b$. Take $$\phi=\left(
                                                      \begin{array}{ccc}
                                                        1 & 0 & 0 \\
                                                        0 & \alpha & \beta \\
                                                        0 & \overline{\alpha} & \overline{\beta} \\
                                                      \end{array}
                                                    \right)\in\operatorname{PGL}_3(\mathbb{C}).$$
 One easily verifies that $\phi\,\,^{\sigma}\phi^{-1}=[X:Z:Y]\in\operatorname{Norm}(G,\operatorname{PGL}_3(\mathbb{C}))$ such that $[X:Z:Y]\,\,A\,\,[X:Z:Y]=A^{-1}$. In particular, we deduce by Theorem \ref{cyclicp2} that $\phi^{-1}\,G\,\phi\leq\operatorname{PGL}_3(\mathbb{R})$ is a model of $G$ over $\mathbb{R}$. More specifically,
\begin{eqnarray*}
    \phi^{-1}\,A\,\phi &=& \left(
                                                      \begin{array}{ccc}
                                                        2\Im(\alpha\,\overline{\beta})\,i & 0 & 0 \\
                                                        0 & \overline{\beta} & -\beta \\
                                                        0 & -\overline{\alpha} & \alpha \\
                                                      \end{array}
                                                    \right)\,\operatorname{diag}(1,\zeta_n^{a},\zeta_n^{-a})\,\left(
                                                      \begin{array}{ccc}
                                                        1 & 0 & 0 \\
                                                        0 & \alpha & \beta \\
                                                        0 & \overline{\alpha} & \overline{\beta} \\
                                                      \end{array}
                                                    \right) \\
     &=& \left(
                                                      \begin{array}{ccc}
                                                        2\Im(\alpha\,\overline{\beta})\,i & 0 & 0 \\
                                                        0 & \zeta_n^a\,\,\overline{\beta} & -\zeta_n^{-a}\,\,\beta \\
                                                        0 & -\zeta_n^a\,\,\overline{\alpha} & \zeta_n^{-a}\,\,\alpha \\
                                                      \end{array}
                                                    \right)\,\left(
                                                      \begin{array}{ccc}
                                                        1 & 0 & 0 \\
                                                        0 & \alpha & \beta \\
                                                        0 & \overline{\alpha} & \overline{\beta} \\
                                                      \end{array}
                                                    \right)\\
     &=&\left(
                                                      \begin{array}{ccc}
                                                       2\Im(\alpha\,\overline{\beta})\,i & 0 & 0 \\
                                                        0 & 2\Im(\alpha\,\overline{\beta}\,\zeta_n^a)\,i & 2|\beta|^2\,\operatorname{sin}(2\pi a/n)\,i \\
                                                        0 & -2|\alpha|^2\,\operatorname{sin}(2\pi a/n)\,i & 2\Im(\alpha\,\overline{\beta}\,\zeta_n^{-a})\,i \\
                                                      \end{array}
                                                    \right)\\
                                                    &=&\left(
                                                      \begin{array}{ccc}
                                                       2\Im(\alpha\,\overline{\beta}) & 0 & 0 \\
                                                        0 & 2\Im(\alpha\,\overline{\beta}\,\zeta_n^a) & 2|\beta|^2\,\operatorname{sin}(2\pi a/n) \\
                                                        0 & -2|\alpha|^2\,\operatorname{sin}(2\pi a/n) & 2\Im(\alpha\,\overline{\beta}\,\zeta_n^{-a}) \\
                                                      \end{array}
            \right)\in\operatorname{PGL}_3(\mathbb{R}).
  \end{eqnarray*}
This completes the proof of Theorem \ref{cor12v}.
\end{proof}

Next, we prove Corollary \ref{cyclicp1v31}.
\begin{proof}(of Corollary \ref{cyclicp1v31}) Assume that $G$ is generated by a non-homology $\pi(A)\in\operatorname{PGL}_3(\mathbb{C})$  of order $n\geq3$.
By Lemma \ref{lem1}, there exists $c\in\mathbb{C}^*$ such that
$$
f_A(t)=(t-c)(t-c\zeta_n^a)(t-c\zeta_n^{b})\in\mathbb{R}[t].
$$
 Moreover, the roots $c,\,c\,\zeta_n^a,\,c\,\zeta_n^b$ of $f_A(t)$ are pairwise distinct, since $\pi(A)$ is a non-homology in $\operatorname{PGL}_3(\mathbb{C})$ by assumption.

Now, the coefficients $c^3\zeta_n^{a+b},\,c(1+\zeta_n^a+\zeta_n^b),\,c^2(\zeta_n^{a+b}+\zeta_n^a+\zeta_n^b)$ belong to $\mathbb{R}$. Thus there are $r,\,r'\in\mathbb{R}$ such that $\zeta_n^{a+b}=r/c^3$ and $\zeta_n^a+\zeta_n^b=r'/c-1$. Consequently, the last condition becomes $c^2(r/c^3+r'/c-1)\in\mathbb{R}$, in other words, $c^3-r'c^2+r''c-r=0$ for some $r,r',r''\in\mathbb{R}$. This means that $c\in\mathbb{C}$ is algebraic over $\mathbb{R}$ of degree dividing $3$.
Since $\mathbb{C}/\mathbb{R}$ is a field extension of degree $2$, then $c$ must be algebraic over $\mathbb{R}$ of degree $1$. Therefore, $c\in\mathbb{R}$, which in turns implies that $\zeta_n^{a+b},\,\zeta_n^a+\zeta_n^b\in\mathbb{R}$.

Clearly, $\zeta_n^{a+b}\in\mathbb{R}$ only if $a+b=k(\frac{n}{2})$ with $k=1,\,2$ or $3$, since $3\leq a+b\leq 2n-3$. If $k=1$ or $3$, then $\zeta_{n}^{a+b}=-1$ and $\zeta_n^a+\zeta_n^b=\zeta_n^a-\zeta_n^{-a}=2\operatorname{sin}(2\pi\,a/n)\,i\notin\mathbb{R}$, a contradiction. Hence $k=1$ and $a+b=0\,\operatorname{mod}\,n$. By Theorem \ref{cor12v} we deduce that $f_A(t)\in\mathbb{R}[t]$ is a sufficient (rather than necessary) condition for $G$ to descend to $\mathbb{R}$.

To see that the converse does not hold in general, take $A=\operatorname{diag}(\zeta_5^3,\zeta_5^4,\zeta_5^2)$ in $\operatorname{GL}_3(\mathbb{C})$. Clearly, $f_A(t)\notin\mathbb{R}[t]$. However, $G=\langle\pi(A)\rangle$ is definable over $\mathbb{R}$ by Theorem \ref{cor12v}, since $\pi(A)=\operatorname{diag}(1,\zeta_5,\zeta_5^{-1})=\operatorname{diag}(1,\zeta_n^a,\zeta_n^b)$ with $n\,|\,a+b$.
\end{proof}
\section{The dihedral group case}\label{sectiondihedral}
Suppose that $G=\langle A,B: A^n=B^2=1,\,BAB=A^{-1}\rangle$ is a dihedral group $D_{2n}$ in $\operatorname{PGL}_3(\mathbb{C})$ with $n\geq3$. There is no loss of generality to take $A=\operatorname{diag}(1,\zeta_n^a,\zeta_n^b)$ up to conjugation and projective equivalence. Since $A$ and $A^{-1}$ are $\operatorname{PGL}_3(\mathbb{C})$-conjugates via $B$, then, by Theorem \ref{cor12v}, $A$ must be a non-homology. Moreover, we can always reduce to the case $b=-a$ modulo $n$ and $B\in\operatorname{PBD}(2,1)$. Since $BAB=A^{-1}$, we obtain $B=[X:\nu Z:\nu^{-1}Y]$ for some $\nu\in\mathbb{C}^*$. Through a projective transformation $\psi=\operatorname{diag}(1,\lambda\nu,\lambda)$, which is in $\operatorname{Norm}\left(\langle A\rangle,\,\operatorname{PGL}_3(\mathbb{C})\right)$, we can further reduce to $\nu=1$. Eventually, we conclude:
\begin{lem}\label{importlema}
For each fixed integer $n\geq3$, there is, up to $\operatorname{PGL}_3(\mathbb{C})$-conjugation, a unique dihedral group $D_{2n}$ of order $2n$. More precisely, any such group is conjugate to the group generated by $B=[X:Z:Y]$ and $A=\operatorname{diag}(1,\zeta_n,\zeta_n^{-1})$.
\end{lem}
Now, we will prove Theorem \ref{dihedralv1}; first by showing that a dihedral grou $G$, as in Lemma \ref{importlema},  always has a real field of moduli, secondly, a model over $\mathbb{R}$ for $G$ exists.
\begin{proof} (of Theorem \ref{dihedralv1})
Since $^{\sigma}\,A\,=\,A^{-1}$ and $^{\sigma}\,B\,=\,B^{-1}$, then $^{\sigma}G\,=\,G$ and { $G$ has a real field of moduli}.

On the other hand, we have seen in Theorem \ref{cor12v} that $\phi^{-1}\,A\,\phi\in\operatorname{PGL}_3(\mathbb{R})$ through a projective transformation $\phi$ of the shape:
$$
\phi=\left(
                                                      \begin{array}{ccc}
                                                        1 & 0 & 0 \\
                                                        0 & \alpha & \beta \\
                                                        0 & \overline{\alpha} & \overline{\beta} \\
                                                      \end{array}
                                                    \right).
$$
It remains to see that $\phi^{-1}\,B\,\phi\in\operatorname{PGL}_3(\mathbb{R})$ so that $\phi^{-1}\,G\,\phi$ is a model of $G$ over $\mathbb{R}$. Indeed, we have

\begin{eqnarray*}
    \phi^{-1}\,B\,\phi &=& \left(
                                                      \begin{array}{ccc}
                                                        2\Im(\alpha\,\overline{\beta})\,i & 0 & 0 \\
                                                        0 & \overline{\beta} & -\beta \\
                                                        0 & -\overline{\alpha} & \alpha \\
                                                      \end{array}
                                                    \right)\,[X:Z:Y]\,\left(
                                                      \begin{array}{ccc}
                                                        1 & 0 & 0 \\
                                                        0 & \alpha & \beta \\
                                                        0 & \overline{\alpha} & \overline{\beta} \\
                                                      \end{array}
                                                    \right) \\
     &=& \left(
                                                      \begin{array}{ccc}
                                                       2\Im(\alpha\,\overline{\beta}) & 0 & 0 \\
                                                        0 & -\beta & \overline{\beta} \\
                                                        0 & \alpha & -\overline{\alpha} \\
                                                      \end{array}
                                                    \right)\,\left(
                                                      \begin{array}{ccc}
                                                        1 & 0 & 0 \\
                                                        0 & \alpha & \beta \\
                                                        0 & \overline{\alpha} & \overline{\beta} \\
                                                      \end{array}
                                                    \right)\\
     &=&\left(
                                                      \begin{array}{ccc}
                                                        2\Im(\alpha\,\overline{\beta})\,i & 0 & 0 \\
                                                        0 & -2\Im(\alpha\,\beta)\,i & -2\Im(\beta^2)\,i \\
                                                        0 & 2\Im(\alpha^2)\,i & 2\Im(\alpha\,\beta)\,i \\
                                                      \end{array}
                                                    \right)\\
                                                    &=&\left(
                                                      \begin{array}{ccc}
                                                        2\Im(\alpha\,\overline{\beta}) & 0 & 0 \\
                                                        0 & -2\Im(\alpha\,\beta) & -2\Im(\beta^2) \\
                                                        0 & 2\Im(\alpha^2) & 2\Im(\alpha\,\beta) \\
                                                      \end{array}
                                                    \right)\in\operatorname{PGL}_3(\mathbb{R}).
  \end{eqnarray*}
This completes the proof of Theorem \ref{dihedralv1}.
\end{proof}

\section{Type $\mathfrak{P}$-  groups}\label{typePgroups}
In this section, we study whether a type $\mathfrak{P}$-group in $\operatorname{PGL}_3(\mathbb{C})$ is definable over $\mathbb{R}$ or not. In other words, we aim to prove Theorem \ref{hessianthm} through the combination of assertions below.

For this, we need the following lemma.
\begin{lem}\label{addedlem1}
  Let $G$ be a finite subgroup in $\operatorname{PGL}_3(\mathbb{C})$. If $\mathcal{G}_9\leq G$, then $G$ is not definable over $\mathbb{R}$.
\end{lem}
\begin{proof}
We first remark that $\mathcal{G}_9$ is isomorphic to $(\Z/3\Z)^2$. Second, we know by \cite[Lemma 5.2]{YHu} (see also \cite[Section 4]{EgorY}) that $(\Z/3\Z)^2$ is a subgroup of $\operatorname{PGL}_3(K)$ if and only if the field $K$ contains a nontrivial cube root of unity. Since $\zeta_3\notin \mathbb{R}$, then  $\mathcal{G}_9$ does not descend to $\mathbb{R}$. That is, we can't find $\phi\in\operatorname{PGL}_3(\mathbb{C})$ for which $\phi^{-1}\mathcal{G}_9\phi\subseteq\operatorname{PGL}_3(\mathbb{R})$. So $G$ is not definable over $\mathbb{R}$ as we wanted to show.
\end{proof}
Next, we tackle each one of these type $\mathfrak{P}$-groups.
\begin{itemize}
\item \textbf{The Hessian groups $\operatorname{Hess}_*:$}
The Hessian group of order $216$, denoted by $\operatorname{Hess}_{216}$, is unique up to conjugation in $\operatorname{PGL}_3(\mathbb{C})$. For instance, we fix $\operatorname{Hess}_{216}=\langle S,T,U,V\rangle$ where
$$S=\operatorname{diag}(1,\zeta_3,\zeta_3^{-1}),\,U=\operatorname{diag}(1,1,\zeta_3),\,T=[Y:Z:X],$$ $$V=\left(
\begin{array}{ccc}
1&1&1\\
1&\zeta_3&\zeta_3^{-1}\\
1&\zeta_3^{-1}&\zeta_3\\
\end{array}
\right).$$
Also, we consider the Hessian subgroup $\operatorname{Hess}_{72}=\langle
S,T,V,UVU^{-1}\rangle$ of order $72$, and the Hessian subgroup
$\operatorname{Hess}_{36}=\langle S,T,V\rangle$ of order $36$.

Concerning the Hessian groups $\operatorname{Hess}_{*}$, for $*\in\{36,\,72,\,216\}$. We show:
\begin{prop}\label{hessians}
Any of the Hessian groups $\operatorname{Hess}_*$ has a real field of moduli.
\end{prop}

\begin{proof}
We have that $^{\sigma}S=S^{-1},\,^{\sigma}U=U^{-1},\,^{\sigma}T=T$, and 
$\,^{\sigma}V=V^{-1}$ in $\operatorname{PGL}_3(\mathbb{C})$. Thus it follows that $^{\sigma}\operatorname{Hess}_*=\operatorname{Hess}_*$ if $*=216$ or $36$. So  $\operatorname{Hess}_{216}$ and $\operatorname{Hess}_{36}$ indeed have a real field of moduli.

For $\operatorname{Hess}_{72}$, we get $^{\sigma}\operatorname{Hess}_{72}=\langle S,T,V, U^{-1}V^{-1}U\rangle\subset\operatorname{Hess}_{216}$; another copy of $\operatorname{Hess}_{72}$ inside $\operatorname{Hess}_{216}$. The \href{https://people.maths.bris.ac.uk/~matyd/GroupNames/193/ASL(2,3).html}{Group structure of $\operatorname{Hess}_{216}$} \cite{groupnames} assures that all copies of $\operatorname{Hess}_{72}$ are $\operatorname{Hess}_{216}$-conjugates, that is to say, there is a projective transformation $\psi\in\operatorname{Hess}_{216}$ such that $\psi^{-1}\,\operatorname{Hess}_{72}\,\psi=\,^{\sigma}\operatorname{Hess}_{72}$. From this we obtain that $\operatorname{Hess}_{72}$ has a real field of moduli as well.
\end{proof}
As a consequence of Proposition \ref{hessians} and Lemma \ref{addedlem1}, we deduce that
\begin{cor}\label{hessianpseudo}
The Hessian groups $\operatorname{Hess}_{*}$, for $*=216, 72, 36$, are pseudo-real.
\end{cor}
\item\textbf{The alternating groups $\operatorname{A}_5$ and $\operatorname{A}_6$}
It is known that $\operatorname{PGL}_3(\mathbb{C})$ possesses a single conjugacy class isomorphic to each of  $\operatorname{A}_5$ and $\operatorname{A}_6$. Therefore, for $*\in\{5,6\}$, $\operatorname{A}_*$  and $^{\sigma}\operatorname{A}_*$ must be $\operatorname{PGL}_3(\mathbb{C})$-conjugates. In other words, $\operatorname{A}_*$ has a real field of moduli.

Since $\operatorname{A}_6$ contains $(\Z/3\Z)^2$ as a subgroup, then we can apply Lemma \ref{addedlem1} to deduce the following.
\begin{cor}
The alternating group $\operatorname{A}_6$ is pseudo-real.
\end{cor}

For the icosahedral group $\operatorname{A}_5$, the situation is a bit  different. To study it  we fix the copy $\mathcal{G}_{360}:=\langle A,\,B,\,C\rangle$ in $\operatorname{PGL}_3(\mathbb{C})$, where
$$
A:=\operatorname{diag}(1,\zeta_5^{-1},\zeta_5),\,B:=[X:Z:Y],$$
$$C:=\left(
                \begin{array}{ccc}
                        2 & 2 & 2 \\
                       1 & \operatorname{cos}(4\pi/5) & \operatorname{cos}(2\pi/5) \\
                        1 & \operatorname{cos}(2\pi/5) & \operatorname{cos}(4\pi/5) \\
                      \end{array}
                    \right).$$
According to \cite[Lemma 2.3.7 ]{Hugg1}, $\mathcal{G}_{360}$ is $\operatorname{PGL}_3(\mathbb{C})$-conjugate to $\operatorname{A}_5$, and any subgroup of $\operatorname{PGL}_3(\mathbb{C})$ isomorphic to $\operatorname{A}_5$ is $\operatorname{PGL}_3(\mathbb{C})$-conjugate to $\mathcal{G}_{360}$.

Now, we are going to construct an explicit model for $\mathcal{G}_{360}$ over $\mathbb{R}$.

\noindent Our study of the Dihedral group in  \S \ref{sectiondihedral} assures that $\langle\,A,\,B\rangle$ descends to $\mathbb{R}$ via a transformation of the shape
$$
\phi=\left(
                                                      \begin{array}{ccc}
                                                        1 & 0 & 0 \\
                                                        0 & \alpha & \beta \\
                                                        0 & \overline{\alpha} & \overline{\beta} \\
                                                      \end{array}
                                                    \right)\in\operatorname{PGL}_3(\mathbb{C}).$$
Moreover, one can check that $\phi^{-1}\,C\,\phi$ equals
$$
\left(
                                                      \begin{array}{ccc}
                                                       4\Im(\alpha\,\overline{\beta}) & 8\Im(\alpha\,\overline{\beta})\,\Re(\alpha) & 8\Im(\alpha\,\overline{\beta})\,\Re(\beta) \\
                                                        2\Im(\overline{\beta}) &2\left(\operatorname{cos}(4\pi/5)\Im(\alpha\overline{\beta})-\operatorname{cos}(2\pi/5)\Im(\alpha\beta)\right) & -2\operatorname{cos}(2\pi/5)\Im(\beta^2) \\
                                                        2\Im(\alpha) & 2\operatorname{cos}(2\pi/5)\Im(\alpha^2) & 2\left(\operatorname{cos}(4\pi/5)\Im(\alpha\overline{\beta})+\operatorname{cos}(2\pi/5)\Im(\alpha\beta)\right) \\
                                                      \end{array}
                                                    \right),
$$
in $\operatorname{PGL}_3(\mathbb{R})$. Thus the generators of $\mathcal{G}_{360}$ when conjugated by the same $\phi$ become in $\operatorname{PGL}_3(\mathbb{R})$, hence the same is true for the whole group and the result follows.

\item\textbf{The Klein group $\operatorname{PSL}(2,7)$} Again, there is a single conjugacy class of $\operatorname{PSL}(2,7)$ in $\operatorname{PGL}_3(\mathbb{C})$, which lead to a real field of moduli. Also, such a class admits a representative that contains the element $\operatorname{diag}(1,\zeta_7,\zeta_7^3)$, cf. \cite[Lemma 2.3.7]{Hugg1}. Applying Theorem \ref{cor12v} with $n=7,\,a=1,\,b=3$, we get that $\operatorname{PSL}(2,7)$ is not definable over $\mathbb{R}$. That is to say,
\begin{cor}
The Klein group $\operatorname{PSL}(2,7)$ is pseudo-real.
\end{cor}
\end{itemize}
\section{Type $\mathfrak{C}$ or Type $\mathfrak{D}$ groups}\label{TypesCD}
Let $G$ be a type $\mathfrak{C}$-group. That is, $G$ is a finite subgroup in $\operatorname{PGL}_3(\mathbb{C})$ generated by $T=[Y:Z:X]$ and $\left\{E_{n,a,b}:=\operatorname{diag}(1,\zeta_n^a,\zeta_n^b)\,:\,(n,a,b)\in I\right\}$, where $I$ is a finite index set.
\begin{lem}
 Let $G$ be a type $\mathfrak{C}$-group as above. Then, it has a real field of moduli.
\end{lem}
\begin{proof}
    One observes that $^{\sigma}G=G$, since $^{\sigma}T=T$ and $^{\sigma}E_{n,a,b}=E^{-1}_{n,a,b}$. In particular, $G$ has a real field of moduli as claimed.
 \end{proof}
\begin{prop}\label{addcase2}
Suppose that $G$ is a type $\mathfrak{C}$-group that is definable over its field of moduli $\mathbb{R}$. Then, one of the following holds.
\begin{enumerate}[(1)]
    \item $G=\langle T\rangle$ or $\mathcal{G}_{12}$, in particular, $G\leq\operatorname{PGL}_3(\mathbb{R})$.
    \item $G=\langle T, E_{m,1,-1}\rangle$; a semidirect product of $\Z/3\Z$ acting on $\Z/m\Z$ for some $m>2$ coprime with $3$.
\end{enumerate}
In particular, the Jordan constant  $J(G)=1$ or $3$.
\end{prop}
\begin{proof}
Assume first  that $n\leq2$ for each triple $(n,a,b)\in I$. Then, $G$ is a subgroup of $\mathcal{G}_{12}\subset\operatorname{PGL}_3(\mathbb{R})$ that contains $T$ of order $3$. Since $\Z/3\Z$ is maximal in $\operatorname{A}_4$, we get that $G=\Z/3\Z$, generated by $T$, or $G=\mathcal{G}_{12}$. Moreover, $J\left(\langle T\rangle\right)=1$ being abelian, and  $J\left(\mathcal{G}_{12}\right)=3$ because $(\Z/2\Z)^2=\langle E_{2,1,0}, E_{2,0,1}\rangle$ is a maximal normal abelian subgroup in $\mathcal{G}_{12}$ with index $3$.

Secondly, assume that $n>2$ for each triple $(n,a,b)\in I$. For $G$ to have a model over $\mathbb{R}$ it is necessary, by Theorem \ref{cor12v}, that $b=-a,\,a=2b$ or $b=2a\,\operatorname{mod}\,n$ for each $(n,a,b)\in I$. Moreover, we always can reduce to the case $E_{n,a,b}=E_{n,a',-a'}$ such that $\operatorname{gcd}(n,a')=1$. For instance, if $b=2a$ then $TE_{n,a,2a}T^{-1}=E_{n,a,-a}\in G$, and if $a=2b$ then $T^{-1}E_{n,2b,b}T=E_{n,-b,b}$. So there is
a finite index set denoted by $I'$ such that
$$G=\langle T, E_{n,a',-a'}\,:\,(n,a')\in I'\rangle\\
=\langle T, E_{m,1,-1}\rangle,$$
where $m=\operatorname{lcm}\left\{n\,:\,(n,a')\in I'\right\}>2.$ Furthermore, if $3$ divides $m$, then $G$ would contain $\mathcal{G}_9=\langle T, E_{3,1,-1}\rangle$. In particular, $G$ does not have a model over $\mathbb{R}$ by the aid of Lemma \ref{addedlem1}, which contradcits our assumption that $G$ is definable over $\mathbb{R}$. Hence, we must have that $3\nmid m$. On the other hand, $G$ acts as an automorphism subgroup for the smooth plane curve
$$
\mathcal{F}_{2m}\,:\,X^{2m}+Y^{2m}+Z^{2m}+X^mY^m+Y^mZ^m+X^mZ^m=0.
$$
We know from \cite[Proposition 6.4]{Harui} that $\operatorname{Aut}(\mathcal{F}_{2m})$ is a semidirect product of $(\Z/m\Z)^2=\langle\operatorname{diag}(1,\zeta_m,1),\,\operatorname{diag}(1,1,\zeta_m)\rangle$ acting on $\operatorname{S}_3=  \langle T,\,R=[X:Z:Y]\rangle$. In particular, $G$ is a semidirect product of $\Z/3\Z=\langle T\rangle$ acting on $\Z/m\Z=\langle E_{m,1,-1}\rangle$. Finally, $G$ is not abelian and $\langle E_{m,1,-1}\rangle$ has index $3$ (the smallest prime dividing $|G|$) in $G$. Therefore,  $\langle E_{m,1,-1}\rangle$ is a maximal abelian normal subgroup of $G$. This in turn implies that $J(G)=3$ in this situation.
\end{proof}
As a result of Proposition \ref{addcase2},
\begin{cor}
Let $G$ be a type $\mathfrak{C}$-group different from $\langle T\rangle$ and $\mathcal{G}_{12}$. Then, $G$ is pseudo-real.     \end{cor}
\begin{proof}
Assume that $G$ is a type $\mathfrak{C}$-group as in
Proposition \ref{addcase2}-(2). Then, $G=\langle T, E_{m,1,-1}\rangle$ with $m>2$ coprime with $3$. Since $\langle E_{m,1,-1}\rangle$ is normal in $G$, we have that $\operatorname{diag}(1,\zeta_m^{-2},\zeta_m^{-1})=T E_{m,1,-1}T^{-1}\in \langle E_{m,1,-1}\rangle.$ That is, there exists a positive integer $l$ such that $\operatorname{diag}(1,\zeta_m^{-2},\zeta_m^{-1})= E_{m,1,-1}^l=\operatorname{diag}(1,\zeta_m^l,\zeta_m^{-l}).$ This holds only if $m\,|\,l+2,\,l-1$. Thus $m\,|(l+2)-(l-1)\,$, and $m=1$, a contradiction. Consequently, $G$ must be as in Proposition \ref{addcase2}-(1).
\end{proof}
Now, let $G$ be a type $\mathfrak{D}$-group. A similar discussion to this of type $\mathfrak{C}$-groups shows that $G$ always has real field of moduli. Moreover, if $G$ descends to $\mathbb{R}$. Then, one of the following holds.
\begin{enumerate}[(1)]
    \item $G=\langle T, R\rangle$ or $\mathcal{G}_{24}$, in particular, $G\leq\operatorname{PGL}_3(\mathbb{R})$.
    \item $G=\langle T, R, E_{m,1,-1}\rangle$; a semidirect product of $\operatorname{S}_3$ acting on $\Z/m\Z$ for some $m>2$ coprime with $3$.
\end{enumerate}
In particular, the Jordan constant  $J(G)=2$ or $6$. Also, we can deduce that any $G$ of the form $(2)$ is pseudo-real.

  \section{Connection to arithmetic geometry}\label{counterexample111}
Let $C:F(X,Y,Z)=0$  be a non-singular plane curve defined over $\mathbb{C}$ with non-trivial automorphism group $\Aut(C)$ in $\operatorname{PGL}_3(\mathbb{C})$,

\begin{lem}\label{automo1}
We have $\operatorname{Aut}(^{\sigma}C)=\,^{\sigma}\operatorname{Aut}(C)$
\end{lem}

\begin{proof}
For any $\phi\in\Aut(C)$, $^{\phi}F(X,Y,Z)=c F(X,Y,Z)$ for some $c\in \mathbb{C}^*$. Applying $\sigma$ to both sides yields
$$
\sigma(c)\,^{\sigma} F(X,Y,Z)=\,^{\sigma}\left(^{\phi}F(X,Y,Z)\right)=\,^{^{\sigma}\phi}\left(^{\sigma}F(X,Y,Z)\right).
$$
That is, $^{\sigma}\phi$ leaves invariant $^{\sigma}C:\,^{\sigma}F(X,Y,Z)=0$. Equivalently, $^{\sigma}\phi\in\Aut(^{\sigma}C)$, hence $^{\sigma}\Aut(C)\subseteq\Aut(^{\sigma}C)$. A similar argument proves the reverse inclusion and the result follows.

\end{proof}
We are now ready to prove Theorem \ref{connection}.

\begin{proof}(of Theorem \ref{connection})
Since { $C:F(X,Y,Z)=0$ has a real field of moduli}, then  $^{\sigma}C:\,^{\sigma}F(X,Y,Z)=0$ and $C:F(X,Y,Z)=0$ are $\mathbb{C}$-projectively equivalent (isomorphic over $\mathbb{C}$). Moreover, any isomorphism between complex non-singular plane curves $C$ and $C'$ is always given by a projective transformation $\phi\in\operatorname{PGL}_3(\mathbb{C})$ such that their automorphism groups are conjugates via this $\phi$. As a consequence, we obtain that $\phi^{-1}\,\operatorname{Aut}(C)\,\phi=\operatorname{Aut}(^{\sigma}C)$, which equals $^{\sigma}\operatorname{Aut}(C)$ by Lemma \ref{automo1}. Thus $\operatorname{Aut}(C)$ has { a real field of moduli} as claimed.

To complete the argument and show that Example \ref{counterexample111} holds, we consider the two-dimensional family $\mathcal{C}_{a,b}$ of smooth plane quintic curves given by
$$
\mathcal{C}_{a,b}: X^5+Y^5+Z^5+iaX(YZ)^2+ibX^3(YZ),
$$
where $a,b\in\mathbb{R}^*$ such that $a/b\neq\dfrac{(c^5-3)c^2}{2c^5-1}\zeta_{10}^{m}$ for any  $c\in\mathbb{C}^*$ and $m\in\{\pm1,\pm3,5\}$.

\begin{itemize}
    \item \textbf{Non-singularity.} We first note that no singular points lie over $Y=0$. Indeed, if $C$ has singularity at $(\alpha:0:\beta)$, then $\alpha$ and $\beta$ must be $0$ in order to satisfy $F_X=F_Z=0$, a contradiction. Second, the resultant of $f_1(X,Z):=F_Y(X,1,Z)$ and $f_2(X,Z):=F_Z(X,1,Z)$ with respect to $X$ is given by $$
    \operatorname{Res}_X(f_1,f_2)=-125\,i\,b^3\,(Z^5-1)^3.
    $$
    Using a computer algebra package such as MATHEMATICA, one can verify that we have singular points over $Z^5=1$ only if $a/b=\dfrac{(c^5-3)c^2}{2c^5-1}\zeta_{10}^{m}$ for some $c\in\mathbb{C}^*$ and $m\in\{\pm1,\pm3,5\}$, which are the cases excluded by assumption.

    \item\textbf{Automorphism group.} The stratification of smooth plane quintics by their automorphism groups in \cite{MR3508302, finalstratum} assures that the group $\operatorname{D}_{10}$ generated by $\rho_1=\operatorname{diag}(1,\zeta_5,\zeta_5^{-1})$ and $\rho_2=[X:Z:Y]$ is always a subgroup of automorphisms for $\mathcal{C}_{a,b}$. Moreover, if $\mathcal{C}_{a,b}$ admits a larger automorphism group, then it should be  $\operatorname{GAP}(150,5)=(\Z/5\Z)^2\rtimes\operatorname{S}_3,$ where in this situation $\mathcal{C}_{a,b}$ is $K$-isomorphic to the Fermat quintic curve $\mathcal{F}_5$; the most symmetric smooth quintic curve. In particular, there must be an extra automorphism $\rho_3\notin\langle\rho_1\rangle$ of order $5$ that commutes with $\rho_1$ as any $\Z/5\Z$ inside $(\Z/5\Z)^2\rtimes\operatorname{S}_3$ is contained in a $(\Z/5\Z)^2$. See \href{https://people.maths.bris.ac.uk/~matyd/GroupNames/129/C5%5E2sS3.html}{Group Structure of $(\Z/5\Z)^2\rtimes\operatorname{S}_3$} \cite{groupnames}.
    Straightforward calculations in $\operatorname{PGL}_3(\mathbb{C})$ lead to $\rho_3=\operatorname{diag}(1,\alpha,\beta)$ with $\alpha^5=\beta^5=1$. Checking the action of such an automorphism on the defining equation of $\mathcal{C}_{a,b}$ tells us that $a=b=0$ or $\rho_3\in\langle\rho_1\rangle$. Therefore, $\Aut(\mathcal{C}_{a,b})=\operatorname{D}_{10}=\langle\rho_1,\rho_2\rangle$.

    Now, we conclude by Theorem \ref{dihedralv1} that $\Aut(\mathcal{C}_{a,b})$ is definable over $\mathbb{R}$.
    \item {\textbf{$\mathcal{C}_{a,b}$ does not have a real field of moduli.} }Suppose that $C$ is a member of the family $\mathcal{C}_{a,b}$ such that {$C$ has a real field of moduli}. Hence $C$ and $^{\sigma}C$ are $\mathbb{C}$-projectively equivalent via some $\phi\in\operatorname{PGL}_3(\mathbb{C})$. Since $C$ and $^{\sigma}C$ belong to the same family $\mathcal{C}_{a,b}$, we have $^{\sigma}\operatorname{Aut}(C)=\operatorname{Aut}(C)=\langle\rho_1,\rho_2\rangle$. In particular, $\phi$ should be in the normalizer of $\langle\rho_1,\rho_2\rangle$ in $\operatorname{PGL}_3(\mathbb{C})$. We reduce to the case  $\phi^{-1}\rho_1\phi=\rho_1$ or $\rho^{-1}$ as $\{c,c\zeta_5,c\zeta_5^{-1}\}\neq\{1,\zeta_5^2,\zeta_5^{-2}\}$ or $\{1,\zeta_5^3,\zeta_5^{-3}\}$ for any $c\in\mathbb{C}^*$ by Lemma \ref{lem1}. Consequently, $\phi=\operatorname{diag}(1,\alpha,\beta)$ or $[X:\alpha Z:\beta Y]$ for some $\alpha,\beta\in\mathbb{C}^*$. Because $^{\phi}C=\,^{\sigma}C$, we must have $\alpha^5=\beta^5=1$ and $\alpha\beta=(\alpha\beta)^2=-1$. The last condition is inconsistent, which means that $C$ and $^{\sigma}C$ are never $\mathbb{C}$-isomorphic.
\end{itemize}
\end{proof}


\end{document}